\documentclass[12pt]{article}
\usepackage{amscd, amssymb, amsmath, amsthm, graphics}
\usepackage{amsmath,amsfonts,amsthm,amssymb}
\usepackage{latexsym,amsmath}
\usepackage{graphicx,psfrag}
\usepackage{amscd, amssymb, amsmath, amsthm, graphics, xy}
\usepackage{amsmath,amsfonts,amsthm,amssymb}
\usepackage{latexsym,amsmath}
\usepackage{graphicx,psfrag}
\input xy
\xyoption{all}

\newtheorem{theorem}{Theorem}[section]
\newtheorem{lemma}[theorem]{Lemma}

\newtheorem{cor}[theorem]{Corollary}
\newtheorem{question}[theorem]{Question}

\theoremstyle{definition}
\newtheorem{definition}[theorem]{Definition}

\theoremstyle{remark}
\newtheorem{remark}[theorem]{Remark}

\numberwithin{equation}{section}

\newcommand{\RR}{{\mathbb R}}
\newcommand{\QQ}{{\mathbb Q}}

\newcommand{\ZZ}{{\mathbb Z}}

\newcommand{\g}{\mathfrak{g}}

\begin{document}
\date{}
\title{Only rational homology spheres \\ admit $\Omega(f)$ to be union of DE attractors}
\author{Fan Ding, Jianzhong Pan, Shicheng Wang and Jiangang Yao}
\maketitle

\begin{abstract}
If there exists a diffeomorphism $f$ on a closed, orientable
$n$-manifold $M$ such that the non-wandering set $\Omega(f)$
consists of finitely many orientable ($\pm$) attractors derived from
expanding maps, then $M$ is a rational homology sphere; moreover all
those attractors are of topological dimension $n-2$.

Expanding maps are expanding on (co)homologies.
\end{abstract}
\maketitle

\tableofcontents

\section{Introduction }
\subsection{Results, motivations and outline of the proof}

Hyperbolic attractors derived from expanding maps, which are also
called as (generalized) Smale solenoids or solenoid attractors,  are
introduced into dynamics by Smale in his celebrated paper \cite{Sm}.
For a diffeomorphism $f$ with Axiom A, which has fundamental
importance in the study of various stabilities in dynamics, Smale
proved the Spectral Decomposition Theorem, which states that
$\Omega(f)$, the non-wandering set of $f$, can be decomposed into
the so-called \textit{basic sets}. He posed several types of basic
sets: \emph{Group 0} which are zero
 dimensional ones
  such as isolated points or the Smale horseshoes; \emph{Group A} and \emph{Group DA},
  both of which are derived from Anosov maps; and
  \emph{Group DE} which are attractors derived from expanding maps.

In the same paper [Sm], Smale posed the following conjecture for
Anosov map.

{\bf Conjecture:} {\it If $f:M\to M$ is an Anosov diffeomorphism,
then the non-wandering set $\Omega(f)=M$.}

For Anosov flows on 3-manifolds, the corresponding problem has a
negative answer given by Franks and Williams [FW]. Inspired by the
construction of [FW], experts in dynamics recognized if one can
construct a dynamics $f:M\to M$ on $2k$-manifold $M$ such that
$\Omega(f)$ consists of one attractor and one repeller which are
derived from expanding maps of type $(k,k)$ (cf. Definitions
\ref{pq} and \ref{pq2}, a repeller of $f$ is an attractor of
$f^{-1}$) with additional transversality condition on stable and
unstable manifolds, the conjecture will also have a negative answer.
This provides a direct motivation to prove the following Theorem
\ref{main1}, which implies
 no dynamics $f$ can make $\Omega(f)$ to be one
 attractor and one repeller derived from expanding maps of type $(k,k)$, at least
in the orientable category (Corollary \ref{main2} (c)).

A basic and elementary fact for a diffeomorphism $f$ on a closed
manifold is that, if the non-wondering set $\Omega(f)$ consists of
finitely many attractors and repellers derived from expanding maps,
then $\Omega(f)$ must consist of exactly one attractor and one
repeller (see [JNW] Lemma 1 for a proof). For brevity we often call
attractors and repellers derived from expanding maps as $(\pm)$
attractors derived from expanding maps.

\begin{theorem}
\label{main1} If there exists a diffeomorphism $f: M \to M$ on a
closed, oriented $n$-manifold $M$ such that $\Omega(f)$ consists of
finitely many oriented $(\pm)$ attractors derived from expanding
maps, then $M$ is a rational homology sphere. Moreover all those
attractors are of type $(n-2, 2)$.
\end{theorem}

\begin{cor}\label{main2}
Let $M$ be a closed oriented $n$-manifold.  If there is a
diffeomorphism $f: M \to M$ such that $\Omega(f)$ consists of
finitely many oriented $(\pm)$ attractors derived from expanding
maps, then

(a) the dimension $n$ of the manifold cannot be 4;

(b) if the dimension $n$ is greater than $4$, those $(\pm)$
attractors  cannot be both derived from expanding maps on tori;

(c) those attractors cannot be both of type $(k, k)$.
\end{cor}

There is a diffeomorphism $f$ on a rational homology $3$-sphere such
that $\Omega(f)$ consists of finitely many $(\pm)$ attractors of
type $(1, 2)$ derived from expanding maps. Indeed recently
\cite{JNW} showed that there exists a diffeomorphism $f: M\to M$ on
a closed orientable 3-manifold $M$ with $\Omega(f)$ a union of
finitely many $(\pm)$ attractors derived from expanding maps if and
only if $M$ is a lens space $L(p,q)$.

We would like to thank Professors L. Wen and C. Bonatti for passing
the information around Smale's conjecture  for Anosov maps once they
heard the work \cite{JNW} and asking us if there is a diffeomorphism
$f: M \to M$ on a closed $4$-manifold $M$ such that $\Omega(f)$
consists of two ($\pm$) attractors derived from expanding maps of
type $(2,2)$.

\bigskip
A more primary motivation for working on Theorem \ref{main1} and
\cite{JNW} is that the information of hyperbolic non-wandering set
$\Omega( f)$ should provide topological information of the manifold
$M$. Both Theorem \ref{main1} and results in \cite{JNW} can be
considered as an analog or  a generalization of the classical result
in mathematics that if there exists a diffeomorphism $f: M \to M$ on
a closed, orientable $n$-manifold $M$ with $\Omega(f)$ a union of
finitely many isolated attractors, then $M$ must be the $n$-sphere.

Stimulated by \cite{JNW} and Corollary \ref{main2} (a), it is
natural to ask
\begin{question}
For which positive integer $n$, there is a dynamics $f: M \to M$ on
a closed $n$-manifold $M$ with $\Omega(f)$ a union of finitely many
$(\pm)$ attractors derived from expanding maps?
\end{question}

To prove Theorem  \ref{main1}, besides to use ideas and methods from
algebraic and geometric topology, as well as dynamics, we also need
the following result, which is of independent interests.

\begin{theorem}\label{main-exp} Let $f: X \to X$ be an
expanding map on the closed oriented manifold $X$, then the induced
homomorphisms $f_*:H_l(X,\RR)\to H_l(X,\RR)$ and $f^*:H^l(X,\RR)\to
H^l(X,\RR)$ are both expanding for any positive integer $l$.
\end{theorem}

\bigskip
Theorem \ref{main-exp} is proved in \S 2 based on two theorems by
Gromov and by Nomizu respectively: any expanding map is conjugate to
an infra-nil-automorphism (Theorem~\ref{class}); for each
nilmanifold $N$, its De Rham cohomology is isomorphic to the
cohomology of the Chevalley-Eilenberg complex associated with the
Lie algebra of the simply connected nilpotent Lie group covering $N$
(Theorem \ref{Nomizu}).

Theorem \ref{main1} is proved in \S 3 with the outline as follows.
 It is known that if
$\Omega(f)$ consists of finitely many ($\pm$) attractors derived
from expanding maps, then $\Omega(f)$ must consist of two such
attractors $S_1$ and $S_2$ with
$$S_1=\bigcap_{h=1}^{\infty} f^h(N_1), \  S_2=\bigcap_{h=1}^{\infty} f^{-h}(N_2),$$
where   $N_1\cong X_1^{p_1} \tilde{\times} D^{q_1}$ and $N_2 \cong
X_2^{p_2} \tilde{\times} D^{q_2}$ are the defining disk bundles of
$S_1$ and $S_2$ respectively. We may choose $N_1$ and $N_2$ so that
$$M= N_1 \cup N_2 , \  P:= N_1 \cap N_2 \ {\rm with}\
\partial P= \partial N_1 \cup \partial
N_2.$$

Consider the Mayer-Vietoris long exact sequence for the pair $(N_1,
N_2):$ $$\label{star} H_l(P) \xrightarrow{\varphi_l=(r_1, r_2)}
H_l(N_1) \oplus H_l(N_2) \xrightarrow{\psi_l= s_1 -s_2} H_l(M).
$$
Since $f|N_1$ descends to an expanding map on $X_1$, $(f|N_1)_*$ is
expanding on $H_l(N_1)$ by Theorem \ref{main-exp}. While $f_*$ must
be an automorphism of $H_l(M)$, thus $s_1$ must be zero
(Lemma~\ref{expend5}). Similarly $s_2$ is zero by considering
$f^{-1}$, and therefore $\psi_l$ must be zero. Consequently
$\varphi_l$ is surjective (Lemma \ref{non-surjective}).

To find the obstruction for the surjectivity of $\varphi_l$, we
first show that the two maps $H_l(\partial N_1)
 \to H_l(P) \leftarrow H_l(\partial N_2)$ induced by the inclusions are
 isomorphisms (Lemma
\ref{twobundles}, which is based on Lemma \ref{exact} and Lemma
\ref{onebundle}).

The disc bundle defining attractor $S_j$ must have zero Euler class,
and with the help of the Gysin sequence, we see that $H_l(\partial
N_j)\cong H_l(X_j) \oplus H_{l-q_j+1}(X_j)$, and there is a subspace
$U$ of $H_l(\partial N_j)$ with $\dim U =\dim H_{l-q_j+1}(X_j)$
which is generated by $S^{q_j-1}$-bundles over cycles in $X_j$,
$j=1,2$ (Lemma~\ref{euler0}).

Note that both $N_1$ and $N_2$ are $K(\pi, 1)$ spaces by
Theorem~\ref{exp}. If $q_1>2$ or $q_2>2$, the image of above
mentioned subspace $U$ under $\varphi_l$ must be zero (Lemma
\ref{shrink}). Then simple dimension analysis indicates that
$\varphi_l$ cannot be surjective. Therefore $q_1=q_2=2$ (cf. Remark
\ref{pq3}) and the attractors must be of type $(n-2, 2)$. Finally
one concludes that $M$ is a rational homology sphere by detailed
homological argument.

\subsection{Definitions related to DE attractors}\label{DE}

\begin{definition}
Let $M$ be a closed manifold, and $f:M\to M$ be a map. An {\em
invariant set} of $f$ is a subset $\Lambda\subset M$ such that
$f(\Lambda)=\Lambda$. A  point $x\in M$ is {\em non-wandering} if
for any neighborhood $U$ of $x$, $f^n(U)\cap U\ne \emptyset $ for
infinitely many integers $n$. Then $\Omega(f)$, the {\em
non-wandering set} of $f$, defined as the set of all non-wandering
points,  is an $f$-invariant closed set. A set $\Lambda\subset M$ is
an {\em attractor} if there exists a closed neighborhood $U$ of
$\Lambda$ such that $f(U)\subset {\rm Int} U$,
$\Lambda=\bigcap_{n=1}^\infty f^n(U)$, and $\Lambda=\Omega(f|U)$.

Now assume $M$ is a Riemannian manifold. A closed invariant set
$\Lambda$ of $f$ is {\em hyperbolic} if there is a continuous
$f$-invariant splitting of the tangent bundle $TM|_\Lambda$ into
{\em stable} and {\em unstable bundles} $E^s_\Lambda\oplus
E^u_\Lambda$ with
\begin{alignat*}{3}
\|Df^m(v)\| &\leq C\lambda^{-m}\|v\|, \quad && \forall \ v\in
E^s_\Lambda,\ &&  m>0,
\\
\|Df^{-m}(v)\| &\leq C\lambda^{-m}\|v\|, \quad && \forall \ v\in
E^u_\Lambda,\ &&  m>0,
\end{alignat*}
for some fixed constants $C>0$ and $\lambda>1$.

A diffeomorphism $f$ is called {\it Anosov} if  $f$ is hyperbolic in
the whole manifold $M$. A map $f$ is called {\it expanding} if $f$
is Anosov and $\dim E^u_\Lambda = \dim M$, in other words, there are
constants $C>0$ and $\lambda>1$ such that
$$\|Df^m(v)\|\geq C\lambda^m\|v\|, \quad \forall v\in TM, \ m >0.$$

Although a metric on $M$ is necessary  to define hyperbolic, Anosov
and expanding maps, whether a map falls into these categories does
not depend on the choice of the metrics because all the norms on
Euclidean spaces are equivalent. It is also shown in [Ma] that for
any expanding map $f$, $M$ has a metric for which $C=1$.
\end{definition}

\begin{definition}
\label{pq}

Let $X$ be the $p$-dimensional compact manifold, $D^q$ be the
$q$-dimensional unit disk, and $N=X \tilde{\times} D^q$, a disk
bundle over $X$. Let $\pi: X \tilde\times D^q \to X$ be the
projection.

Suppose an embedding $e: X\tilde \times D^q \to X\tilde \times D^q$
satisfies the following two conditions:

(1) For some expanding map $\varphi: X \to X$, the diagram
$$\begin{CD}
 X \tilde\times D^q @>e>> X \tilde\times D^q \\
 @V\pi VV   @VV\pi V \\
X @>>\varphi > X \end{CD}$$commutes, i.e. $\varphi \circ \pi =\pi
\circ e$, or $e$ preserves the disk fibers and descends to the
expanding map $\varphi$ by $\pi$,

(2) $e$ shrinks the $D^q$ factor evenly by some constant $0<\lambda<
1$, i.e. the image of each fiber $D^q$ under $e$ is $q$-dimensional
disk with radius $\lambda$,

then we call $e$ a {\it hyperbolic bundle embedding} and $S =
\bigcap_{n\ge0}^{\infty} e^n(N)$ an {\it attractor of type $(p,q)$}
derived from expanding map $\varphi$.
\end{definition}

\begin{definition}\label{pq2}
Let $M$ be a $(p+q)$-manifold and $f: M\to M$ be a diffeomorphism.
If there is an embedding $N \cong X^p\tilde \times D^q \subset M$
such that $f|N$ (resp.\ $f^{-1}|N$) conjugates $e:N\to N$, where $e$
and $N$ are as given in Definition \ref{pq}, we call
$S=\bigcap_{h=1}^{\infty} f^{h}(N)$ (resp.\
$S=\bigcap_{h=1}^{\infty} f^{-h}(N)$) an (+) {\em attractor (resp.
$(-)$ attractor) of type $(p,q)$ derived from expanding maps} , and
we also say $N$ is a \textit{defining disk bundle} of $S$.

For simplicity, call both ($+$) attractor and  ($-$) attractor
derived from expanding maps as {\it DE attractors}. Call a DE
attractor is {\it oriented} if the base manifold $X$ is oriented.
Call a DE attractor is {\it toric} if $X$ is a torus.

\end{definition}

\begin{remark}\label{pq3} DE attractors of type $(p,q)$
is homeomorphic to  so called {\it $p$-dimensional solenoids} for
any $q$. DE attractors of type (1,2),  nested intersections of solid
tori $S^1 \times D^2$, is the traditional solenoids introduced by
Vietoris in 1927 [Vi].  The DE attractors given in Smale's paper
[Sm] are of type $(p, p+1)$, where the condition $q=p+1$ guarantees
any expanding map on base manifold can be lifted to an embedding of
the disk bundle in the construction. On the other hand, DE
attractors of type $(p,1)$ do not exist for any $p$: there is no
embedding $e: X \times D^1 \to X \times D^1$ lifted from an
expanding $\varphi: X\to X$; indeed no $p$-dimensional solenoids
embed into $(p+1)$-dimensional manifolds \cite{JWZ}.
\end{remark}

Recall that two maps $f: X\to X$ and $g: Y \to Y$ are called
\textit{topological conjugate} if there exists a homeomorphism $h: X
\to Y$ such that $h \circ f = g \circ h$.

Except noted otherwise, all homology and cohomology groups appeared
in this paper are over the field of real numbers $\RR$, i.e.
$H_*(X)$ means $H_*(X; \RR)$, $H^*(X)$ means $H^*(X; \RR)$. The
$i$-th Beti number, i.e.  $\dim H_i(X; \RR)$ is denoted by
$\beta_i(X)$.

\section{Expanding maps}

\subsection{Expanding linear maps}

\begin{definition} Let $V$ be a finite dimensional vector space over $\RR$,
$f:V \to V$ is a linear map, if all the eigenvalues of $f$ are
greater than $1$ in absolute value, then $f$ is called
\textit{expanding}. For a real square matrix $A$, if all of its
eigenvalues are greater than $1$ in absolute value, then $A$ is
called \textit{expanding}.
\end{definition}

\begin{lemma} \label{suj-exp} Let $V$ and $W$ be two finite dimensional real vector spaces,
$f$, $g$ and $h$ are linear maps that make the following diagram
commute.
$$\begin{CD}
 V @>f>> V \\
 @VhVV   @VVhV \\
W @>>g> W \end{CD}$$ If $f$ is expanding and $h$ is surjective, then
$g$ is expanding.
\end{lemma}

\begin{proof} Suppose $\dim V=n$, $\dim W= m$. Let $w_1, \cdots, w_m$ be a basis for $W$, by the
surjectivity of $h$, choose $v_1, \cdots, v_m \in V$ such that
$h(v_i)= w_i$, $ 1\leq i \leq m$. Clearly $v_1, \cdots, v_m$ are
linearly independent, we may then find more vectors $v_{m+1},
\cdots, v_n$ in $\ker h$ so that $v_1 , \cdots, v_n$ is a basis for
$V$. Suppose $f(v_i) =\sum_{j=1}^{n} c_{ij} v_j$.  Then for $m+1
\leq i \leq n$, we have
$$0 = g(0)= g(h(v_i))= h(f(v_i)) = h(\sum_{j=1}^{n} c_{ij} v_j)=
\sum_{j=1}^{m} c_{ij} w_j.$$ Thus $c_{ij}=0$ for $m+1 \leq i \leq n$
and $1 \leq j \leq m$. Under the above two bases, the matrix
representation of $f$ is $A= \begin{bmatrix} A_{11} & 0 \\
A_{21} & A_{22}
\end{bmatrix}$, where $A_{11}$ and $A_{22}$ are $m \times m$ and $(n-m)
\times (n-m)$ matrices respectively. This matrix $A$ is expanding,
so all of its eigenvalues are greater than 1 in absolute value. Note
that all the eigenvalues of $A_{11}$ are eigenvalues of $A$. Hence
$A_{11}$ is expanding. It follows from $h(v_i)= w_i$ that linear map
$g$ corresponds to matrix $A_{11}$. Therefore $g$ is an expanding
map.
\end{proof}

\begin{lemma} \label{dual-exp} Let $f: V \to V$ be a linear map on the finite dimensional real vector
space $V$. Denote the dual map of $f$ by $f^*: V^* \to V^*$, where
$V^*={\rm Hom}(V, \RR)$.

If $f$ is expanding, so is $f^*$ and vice versa.\end{lemma}

\begin{proof} Fix a basis for $V$ and a dual basis for $V^*$, then
$f$ and $f^*$ are represented by two square matrices $A$ and $A^T$
respectively. The claim then follows from the fact that $A$ and
$A^T$ have the same characteristic polynomials, and thus the same
eigenvalues.
\end{proof}

\begin{lemma} \label{ext-exp} Let $f: V \to V$ be an expanding linear
map on the real vector space $V$, for any positive integer $i$,

(a) the induced map $\otimes^i f : \otimes^i V \to \otimes^i V$ is
expanding,

(b) the induced map $\wedge^i f: \wedge^i V \to \wedge^i V$ is
expanding.
\end{lemma}

\begin{proof} (a) There is a standard fact in multilinear algebra.

Let $f: V \to V$ and $g: W \to W$ be linear maps with $\dim V=n$ and
$\dim W=m$. Suppose $\lambda_1, \cdots, \lambda_n$ are the
eigenvalues of $f$ and $\mu_1, \cdots, \mu_m$ are the eigenvalues of
$g$, then the eigenvalues of $f\otimes g: V \otimes W \to V \otimes
W$ are $\lambda_i\mu_j$, where $i=1,\cdots, n$, and $j=1, \cdots,
m$.

Thus each eigenvalue of $\otimes^i f$ is a product of $i$ complex
numbers with absolute value greater than $1$, which is still greater
than $1$ in absolute value. Therefore, by definition, $\otimes^i$ is
expanding.

(b) Note that there is a surjective map from $\otimes^i V$ to
$\wedge^i V$, then it follows from Lemma~\ref{suj-exp} that
$\wedge^i f$ is also expanding.
\end{proof}

\begin{lemma}\label{chain-exp} Let $(C_*,d)$ be a cochain complex,
where each $C_i$ is a finite dimensional real vector space, and $f:
(C_*, d) \to (C_*,d )$ be a chain map. Suppose $f_i: C_i \to C_i$ is
expanding for all positive integer $i$, then $f$ induces expanding
maps on the cohomology groups $H^i(C)$ for $i>0$.
\end{lemma}

\begin{proof} We label the boundary maps in the cochain complex as
$d_i: C_i \to C_{i+1}$, then $$H^i(C)=\frac{{\rm Ker} \ d_i}{{\rm
Im} \ d_{i-1}}.$$ As $f$ is a chain map,  ${\rm Ker} \ d_i$ is a
$f_i$-invariant subspace of $C_i$.

Fix a positive integer $i$, choose a basis for ${\rm Ker} \ d_i$ and
then extend it to a basis for $C_i$. Clearly, under this basis, the
matrix representation of $f_i$ takes the form $A=\begin{bmatrix}
A_{11} & A_{12}
\\0 & A_{22}
\end{bmatrix}$. Since $f$ is expanding, all the eigenvalues of $A$
are greater than $1$ in absolute value, so are eigenvalues of
$A_{11}$. Therefore, the restriction of $f_i$ on ${\rm Ker} \ d_i$
is expanding. Now we have a commutative diagram
$$\begin{CD}
{\rm Ker} \ d_i @>f_i>> {\rm Ker} \ d_i \\
 @VVV   @VVV \\
H^i(C) @>f^{*}_i>> H^i(C)
\end{CD}$$
Clearly, the two vertical maps are surjective, then according to
Lemma~\ref{suj-exp}, $f^{*}_i$ on $H^i(C)$ is expanding as well.
\end{proof}

\subsection{Gromov's theorem and Nomizu's theorem}

The most basic manifolds that admit expanding maps are $n$-tori.
Expanding maps are systematically discussed in late 1960's. We quote
some major properties as follows.

\begin{theorem}
\label{exp} Let $M$ be a closed $n$-manifold and $f: M\to M$ be an
expanding map. Then

(a) The universal cover of $M$ is  $\RR^n$ and $f$ is a covering map
(Shub [Sh1]).

(b) Any flat manifold admits expanding maps (Epstein-Shub [ES]).

(c) If $g: M\to M$ is an expanding map which is homotopic to $f$,
then $f$ and $g$ are topologically conjugate (Shub [Sh2]).

(d) $\pi_1(M)$ has polynomial growth (Franks [Sh2]).
\end{theorem}

In [Gr], Gromov proved that if a finitely generated group $G$ has
polynomial growth, then $G$ contains a nilpotent subgroup of finite
index. Combining this result with Theorem~\ref{exp} (c) and (d), he
obtained

\begin{theorem} \label{class} An expanding map on a closed manifold is topologically
conjugate to an infra-nil-endomorphisms.
\end{theorem}

Recall that a \textit{nilpotent} Lie group $G$ is a connected Lie
group whose Lie algebra $\g$ is a \textit{nilpotent} Lie algebra.
That is, its Lie algebra lower central series
$$\g_1=[\g,\g], \ \g_2=[\g, \g_1], \cdots, \g_{k+1}=[\g, \g_{k}],\
\cdots$$ eventually vanishes.

Suppose $G$ be a simply connected nilpotent Lie group. Denote by
${\rm Aff}(G)$ the group of transformations of $G$ generated by all
the left translations and all the automorphisms on $G$. In other
words, ${\rm Aff}(G)= G \ltimes {\rm Aut}(G)$.

For a discrete subgroup $\Gamma$ of $G$,  call $G/\Gamma$ a
\textit{nilmanifold} when it is compact. For a group $\Gamma \subset
{\rm Aff}(G)$ which acts freely and discretely on $G$, call
$N=G/\Gamma$ \textit{infra-nil-manifold} when it is compact. Every
automorphism of $G$ which respects the group action $\Gamma$ induces
a map on $G/\Gamma$. Such maps are called
\textit{infra-nil-endomorphism}.

\begin{theorem}\label{Nomizu} For each
nilmanifold $N=G/\Gamma$, the De Rham cohomology of $N$ is
isomorphic to the cohomology of the Chevalley-Eilenberg complex
$(\wedge \g^*, \delta)$ associated with the Lie algebra $\g$ of $G$,
that is
$$H^*(N;\RR) \cong H^* (\wedge \g^*, \delta).$$ \end{theorem}

%In [CE], Chevalley-Eilenberg complex was introduced for Lie algebra
%$\g$.

In [CE], Chevalley-Eilenberg complex $(\wedge \g^*, \delta)$  was
constructed for Lie algebra $\g$ as follows. Let $\{X_1, \cdots,
X_s\}$ be the basis for $\g$, and $\{ x_1, \cdots, x_s \}$ be the
dual basis for $\g^*$, the dual of $\g$. On the exterior algebra
$\wedge \g^*$, construct the differential $\delta$ by defining it on
degree 1 elements as
$$\delta x_k (X_i, X_j)= - x_k ([X_i, X_j]),$$
and extending to $\wedge \g^*$ as a graded bilinear derivation.
Suppose $[X_i, X_j]=\sum_l c_{ij}^{l}X_{l}$, where $c_{ij}^{l}$ are
the structure constants of $\g$, then $\delta x_k(X_i,
X_j)=-c_{ij}^{k}$, and thus the differential $\delta$ on generators
has the form
$$\delta x_k = - \sum_{i<j} c_{ij}^{k} x_i \wedge x_j.$$
Now it is clear that the Jacobi identity in the Lie algebra is
equivalent to the condition $\delta^2=0$. Hence $(\wedge \g^*,
\delta)$ is a cochain complex.

For Lie group $G$, we consider the cochain complex consisting of all
the $G$-right invariant differential forms on $G$ over $\RR$ with
ordinary differential $d$, denote it by $(C(G), d)$. From the
definition of Lie algebra associated with a Lie group, we know
$$\label{iso1} (\wedge \g^*,\delta) \cong (C(G), d).
$$

For $N=G/\Gamma$, Let $C(N)$ be the cochain complex consisting of
all the $\Gamma$-right invariant differential forms on $G$ over
$\RR$, denote it by $(C(N),d)$. It is clear that $H^*(C(N))$ is the
De Rham cohomology of the nilmanifold $N$, that is $$\label{iso2}
H^*(N;\RR) \cong H^*(C(N))
$$

Under the condition that $G$ is nilpotent, Nomizu [No] showed that
the natural inclusion $C(G) \to C(N)$ induces an isomorphism on the
cohomology level $$\label{iso3} H^*(C(N)) \cong H^*(C(G))
$$

Theorem~\ref{Nomizu} then follows from the above three isomorphisms.

\subsection{Expanding maps are expanding on (co)homology groups}

Now we are ready to prove

\noindent \textbf{Theorem \ref{main-exp}} \  \textit{Let $f: X \to
X$ be an expanding map on the closed oriented manifold $X$, then the
induced homomorphisms $f_*:H_l(X,\RR)\to H_l(X,\RR)$ and
$f^*:H^l(X,\RR)\to H^l(X,\RR)$ are both expanding for any positive
integer $l$. }

Before the proof, we need a lemma about the induced homomorphisms on
homology groups of covering and expanding maps.

\begin{lemma}\label{expandhomo}
(a) If $f: X\to Y$ is a covering map between the oriented closed
manifolds, then the induced map $f_*: H_*(X)\to H_*(Y)$ is an
epimorphism.

(b) If $f$ is an expanding map on the oriented closed manifold $X$,
then the induced map $f_*: H_*(X)\to H_*(X)$ is an isomorphism.
\end{lemma}

\begin{proof} (a) $f: X\to Y$ is a covering between the oriented closed
manifolds implies that $f$ is a map of non-zero degree. Then it is
well-known that $f_*$ is surjective (cf. Lemma 7 in [WZ]).

(b) $f: X\to X$ is expanding implies that $f$ is a covering. $X$ is
closed and orientable implies that $f$ is a map of non-zero degree.
Hence  $f_*$ is surjective on $H_*(X) $ by (a). A self-surjection on
any finite dimensional real vector space must be an isomorphism.
\end{proof}

\medskip

\noindent \textbf{Proof of Theorem \ref{main-exp}.} As topologically
conjugate maps induce conjugate homomorphisms on (co)homology
groups, by Theorem~\ref{class}, we may assume that $f$ is an
infra-nil-endomorphism.

%Suppose $f:X\to X$ is an expanding infra-nil-endomorphism, then
%according to the definition, $X= G / \Gamma$ for some simply
%connected nilpotent Lie group $G$ and $\Gamma \subset {\rm Aff}(G)$,
%and $f$ can be lifted to an automorphism $A$ of $G$ which respects
%$\Gamma$. $A$ induces a linear map $a: \g \to \g$ in the Lie algebra
%$\g$ of G. $f$ is expanding implies $A$ is expanding and
%consequently $a$ is an expanding linear map. Let $\Gamma'= G \cap
%\Gamma$, then $N=G/\Gamma'$ is a nilmanifold. Clearly $N$ is a
%covering of $X$ and $f$ can be lifted to an expanding map $g:N \to
%N$. Thus we have the following commutative diagram

Suppose $f:X\to X$ is an expanding infra-nil-endomorphism, then
according to the definition, $X= G / \Gamma$ for some simply
connected nilpotent Lie group $G$, $\Gamma \subset {\rm Aff}(G)$ is
a discrete uniform subgroup, and $f$ can be lifted to an
automorphism $A$ of $G$ which respects $\Gamma$. $A$ induces a
linear map $a: \g \to \g$ in the Lie algebra $\g$ of G. $f$ is
expanding implies $A$ is expanding and consequently $a$ is an
expanding linear map. Let $\Gamma'= G \cap \Gamma$, then
$\Gamma/\Gamma'$ is a finite group and  $N=G/\Gamma'$ is a
nilmanifold (see  \cite{AA} Theorem 1). Clearly $N$ is a covering of
$X$ and $f$ can be lifted to an expanding map $g:N \to N$. Thus we
have the following commutative diagram

$$\begin{CD}
G  @>A>> G \\
 @Vp_1VV   @VVp_1V \\
N @>g>> N \\
@Vp_2VV    @VVp_2V \\
X @>f>> X \end{CD}$$ where both $p_1$ and $p_2$ are covering maps.

Now we will show that $g^*: H^l(N; \RR) \to H^l(N; \RR)$ is
expanding for any positive integer $l$. By Theorem~\ref{Nomizu},
$H^*(N;\RR) \cong H^* (\wedge \g^*, \delta).$ Note that this
isomorphism is natural and $g$ descends from $A$ or $a$. Thus the
expanding property of $g^*$ is equivalent to that
$$a^*: H^l(\wedge \g^*) \to H^l(\wedge \g^*)$$ is expanding for
any $l>0$.

By the construction of $A$, we know that $a$ is a linear expanding
map on the Lie algebra $\g$. Then by Lemma~\ref{dual-exp}, $a^*:
\g^* \to \g^*$ is also expanding. Next it follows from
Lemma~\ref{ext-exp} that $a^*: \wedge^l \g \to \wedge^l \g$ is
expanding for any $l>0$. Thus according to Lemma~\ref{chain-exp},
$a$ induces expanding homomorphisms on the cohomology level.

Therefore we have proved that $g^*$ is expanding on $H^l(N;\RR)$ for
any positive integer $l$. Note that $H^*(N; \RR) ={\rm Hom} (H_*(N;
\RR); \RR)$, by Lemma~\ref{dual-exp}, we know $g_*: H_l(N; \RR) \to
H_l(N; \RR)$ is also expanding for $l>0$.

Taking homology on the above diagram, we get the commutative diagram
$$\begin{CD}
H_l(N; \RR)  @>g_*>> H_l(N; \RR) \\
 @Vp_{2*}VV   @VVp_{2*}V \\
H_l(X; \RR) @>f_*>> H_l(X; \RR) \\
\end{CD}$$ By Lemma~\ref{expandhomo}, $p_{2*}$ is an
isomorphism, then it follows from Lemma~\ref{suj-exp} that $f_*$ is
an expanding map for $l>0$. At last, according to
Lemma~\ref{dual-exp}, $f^*: H^l(X; \RR) \to H^l(X; \RR)$ is
expanding as well. \hfill $\square$

\section{Proof of the main theorem}

\subsection{Maps between abelian groups}

The first lemma is a useful fact about exact sequence in homological
algebra. Except the elementary proof given here, it can also be
proved in the language of category theory by using pull-back
diagram. See Theorem II 6.2 and Lemma III 1.1 \& 1.2 in [HS] for
details.

\begin{lemma}
\label{exact} Let $A$, $B_1$, $B_2$ and $C$ be abelian groups and
$$A \xrightarrow{\varphi=(i_1, i_2)} B_1 \oplus
B_2 \xrightarrow{\psi= j_1 -j_2} C $$ be an exact sequence, then

(a) \ $i_1$ is injective $\Rightarrow$ $j_2$ is injective.

(b) \ $i_1$ is injective $\Leftarrow$ $j_2$ is injective, provided
that $\varphi$ is injective.

(c) \ $i_1$ is surjective $\Rightarrow$ $j_2$ is surjective,
provided that $\psi$ is surjective.

(d) \ $i_1$ is surjective $\Leftarrow$ $j_2$ is surjective.

(e) All the above claims are still hold if we substitute $i_1$ and
$j_2$ by $i_2$ and $j_1$ respectively.

\end{lemma}

\begin{proof} Each statement can be shown by diagram chasing, for
example, we prove (a) and (c) as follows.

(a) Suppose $j_2(b_2)=0$, then $\psi(0,b_2)=0$. So $(0,b_2) \in {\rm
Im} \varphi$, we can find $a \in A$ such that $\varphi(a)=(0,b_2)$.
It follows from the injectivity of $i_1$ that $a=0$. Thus
$b_2=i_2(a)=0$, which implies $j_2$ is injective.

(c) For any $c \in C$, as $\psi$ is surjective, we may find $(b_1,
b_2) \in B_1 \oplus B_2$ such that $\psi(b_1,b_2)=c$. By the
surjectivity of $i_1$, pick up $a \in A$ such that
$\varphi(a)=(-b_1, b_2')$. Now $j_2(b_2+b_2')= \psi(0, b_2+b_2')=
\psi(b_1,b_2)+ \psi (-b_1, b_2')= c+ \psi \circ \varphi(a) =c$,
which means $j_2$ is surjective.

%(d) Suppose $j_2$ is surjective. For any $b_1 \in B_1$, by the
%sujectivity of $j_2$, we may find $b_2 \in B_2$ such that
%$j_2(b_2) = j_1(b_1)$. Thus $(b_1, b_2) \in {\rm Ker} \psi = {\rm
%Im} \varphi$. So some $a \in A$ makes $i_1(a)=b_1$, $i_1$ is
%surjective.

(e) is clear by symmetry.
\end{proof}

\begin{lemma}\label{expend5} Let $A,B$ be two finitely generated free abelian groups.
Let $f:A\to A$ be an expanding homomorphism. Let $g:B\to B$ be an
automorphism. Then there does not exist a non-zero homomorphism
$h:A\to B$ such that $h\circ f=g\circ h$.
\end{lemma}

\begin{proof} Assume that there exists a non-zero homomorphism $h:A\to
B$ such that $h\circ f=g\circ h$. The image of $h$, $h(A)$, is a
subgroup of $B$. Thus it is a finitely generated free abelian group
and there exists a basis $e_1,\ldots, e_r$ of $B$ and non-zero
integers $a_1,\ldots, a_s$ such that $a_1e_1,\ldots,a_se_s$ form a
basis of $h(A)$. Since $f$ is expanding and $h\circ f=g\circ h$, the
restriction of $g$ to $h(A)$, $g|h(A):h(A)\to h(A)$, is expanding
(using the fact that $A$ is isomorphic to the direct sum of $\ker h$
and $h(A)$). Let $C$ denote the subgroup of $B$ spanned by
$e_1,\ldots,e_s$. Since $g$ maps $h(A)$ into $h(A)$, $g$ maps $C$
into $C$ and the restriction of $g$ to $C$, $g|C:C\to C$, is also
expanding. The matrix $M$ of $g$ with respect to the basis
$e_1,\ldots,e_r$ looks like
$$\left ( \begin{array}{cc} M_1 & * \\ 0 & M_2 \end{array} \right
),$$ where $M_1$ is an $s\times s$ integer matrix and $M_2$ is an
$(r-s)\times (r-s)$ integer matrix. Since $g|C:C\to C$ is expanding,
the absolute value of the determinant of $M_1$ is greater than $1$.
Since $M_2$ is an integer matrix and non-degenerate ($g$ is an
automorphism, thus $M$ is non-degenerate), the absolute value of the
determinant of $M_2$ is greater than or equal to $1$. Thus the
absolute value of the determinant of $M$ is greater than $1$. But
since $g$ is an automorphism, the absolute value of the determinant
of $M$ equals $1$, a contradiction. \end{proof}

\subsection{Related results in algebraic topology}
%Most of the techniques in homology theory involved thereafter can be
%found in standard textbooks such as \cite{Sp}.

%\cite{Sp}

% For Gysin sequence and Thom isomorphism in Lemma
%\ref{euler0}, \cite{Sp} (pp 255-260) is a good reference .

Most of the techniques in algebraic topology involved thereafter can
be found in standard textbooks such as \cite{Sp}. However for the
reader¡¯s convenience, we will recall some of them:

\begin{theorem} (Exact sequence for pair $(X,A)$)
Let $A$ be a subspace of  a topological space $X$. there is a
(co)homology long exact sequence associated to the pair $(X,A)$:
$$\cdots \rightarrow  H_l(A)
\stackrel{i}\rightarrow H_l(X) \stackrel{j} \rightarrow H_l(X, A)
\stackrel{\partial} \rightarrow  H_{l-1}(A) \stackrel{i}\rightarrow
H_{l-1}(X)\cdots$$
$$\cdots \rightarrow  H^l(X,A)
\stackrel{j}\rightarrow H^l(X) \stackrel{i} \rightarrow H^l(A)
\stackrel{\delta} \rightarrow  H^{l+1}(X,A) \stackrel{j}\rightarrow
H^{l+1}(X)\cdots$$
\end{theorem}

\begin{theorem} (Mayer-Vietoris sequence)
Let $X$ be a topological space and $A, B$ be two subspaces whose
interiors cover $X$. The Mayer-Vietoris sequence in singular
homology for the triad $(X, A, B)$ is a long exact sequence relating
the singular homology groups  of the spaces $X, A, B$, and the
intersection $A \cap B$.
$$\cdots \rightarrow  H_{l+1}(X)
\stackrel{\partial}\rightarrow H_l(A \cap B) \stackrel{(i_1,i_2)}
\longrightarrow H_l(A) \oplus H_l(B) \stackrel{j_1-j_2}
\longrightarrow H_{l}(X) \stackrel{\partial_*}\rightarrow H_{l-1}(A
\cap B) \cdots$$ where the homomorphisms $i_1,i_2,j_1,j_2$ are
induced from the inclusions  $A\cap B \subset A$, $A \cap B  \subset
B$, $A\subset X$ and $B \subset X$ respectively.
\end{theorem}

Let  $N \stackrel{\pi}\to X$  be a $q$-disk bundle and $ \partial N
\stackrel{\pi}\to X$ be the associated $(q-1)$-sphere bundle. Then a
Thom class for the bundle is an element $t \in H^q(N,\partial N)$
such that the restriction of $t$ to each fiber is non zero. If the
Thom class exists, the disk bundle is called orientable. The Euler
class  $e$ of the disk bundle or the associated sphere bundle is the
image of $t$ under the map $ H^q(N,\partial N) \stackrel{j}
\rightarrow  H^{q}(N) \stackrel{({\pi}^*)^{-1}} \rightarrow H^q(X)$.

\begin{theorem}(Naturality of Euler class) If $f:X\to X'$ is covered by
an orientation preserving bundle map $\xi\to\xi'$, then
$e(\xi)=f^*e(\xi')$.
\end{theorem}

\begin{theorem} (Thom isomorphism theorem)
Let  $N \stackrel{\pi}\to X$  be an orientable $q$-disk bundle and $
\partial N \stackrel{\pi}\to X$ be the associated $(q-1)$-sphere
bundle with Thom class $t \in H^q(N,\partial N)$. Then the following
homomorphisms are isomorphisms for all $l \in Z$
$$\Phi^*: H^l(X)  \to H^{l+q}(N, \partial N)$$
$$\Phi_*: H_{l+q}(N, \partial N)\to H_l(X) $$
where $\Phi^*(z)=\pi^*(z) \cup t$ and $\Phi_*(x)=\pi_*(t \cap x)$
\end{theorem}

From the Thom isomorphism theorem one can construct the Gysin
sequence as follows:

Consider the commutative diagram
$$\begin{CD} H_{l+1}(N, \partial N) @>\partial>> H_l(\partial N) @>>> H_l(N) @>>>
H_l(N, \partial N) @ >\partial>> H_{l-1}(\partial N) \\
 @VV\Phi_*V   @| @VV\pi_*V @VV\Phi_*V @|\\
H_{l-q+1}(X) @>>> H_l(\partial N) @>>> H_l(X) @>\cap e >> H_{l-q}(X)
@>>> H_{l-1}(\partial N)\end{CD}.$$ The first line is the long exact
sequence for the pair $(N, \partial N)$. The vertical map $\pi_*$ is
an isomorphism as $N\simeq X$. The vertical map $\Phi_*$ is also an
isomorphism, called \textit{Thom isomorphism}. Therefore the second
line is also a long exact sequence, which is named as \textit{Gysin
sequence}. Moreover the map from $H_l(X)$ to $H_{l-q}(X)$ is defined
by $x\mapsto  e\cap x$, where $e$ is the Euler class of the disk
bundle $\xi$, $\cap$ means cap product.

\subsection{Topology of sphere bundles}

\begin{lemma} \label{euler0} Let $\varphi$ be an expanding map on the closed oriented manifold
$X$ and $\xi$ be an oriented disk bundle $N=X \tilde{\times} D^q
\stackrel{\pi}\to X$ such that $\varphi$ can be lifted to a
hyperbolic bundle embedding $e$ on $N$. Then

(a) the Euler class of $\xi$, $e(\xi)=0 \in H^q(X)$;

(b) $H_l(\partial N) \cong H_l(X) \oplus H_{l-q+1}(X)$;

(c) Let $\{ [c_1], [c_2], \cdots\}$ be a basis for $H_{l-q+1}(X)$,
where each $c_i$ is a cycle in $X$, then $[(\pi|\partial
N)^{-1}(c_1)], [(\pi|\partial N)^{-1}(c_2)], \cdots$ are linearly
independent in $H_l(\partial N)$. As a consequence, their span is a
subspace $U \subset H_l(\partial N)$ with $\dim U= \dim
H_{l-q+1}(X)$.
\end{lemma}

\begin{proof} (a) We may assume that $e$ is orientation preserving, otherwise we use
$e^2$ instead of $e$. The embedding $e: X \tilde{\times} D^q \to X
\tilde{\times} D^q$ is not a bundle map, however we can modify it to
become a bundle map. Since $\xi$ may not be a trivial bundle, we use
local presentation of $\xi$. For each $x_0\in X$, there is a
neighborhood $U$ of $x$ such that $\xi|U$ is a trivial bundle, that
is, $\xi|U\cong U\times D^q$. Similarly, there is a neighborhood $V$
of $\varphi(x_0)$ such that $\xi|V$ is a trivial bundle, that is,
$\xi|V\cong V\times D^q$. Possibly shrinking $U$, we may assume
$\varphi(U)\subset V$. According to Definition 1.6, we may assume on
$\xi|U\cong U\times D^q$,
$$e(x,y)=(\varphi (x), c(x)+\lambda r(x)y), x\in U, y\in D^q,$$
where $\lambda\in (0,1)$, $c(x)$ and $r(x)$ are functions from $U$
to $D^q$ and $SO(q)$ respectively with $|c(x)|+\lambda <1$. The the
map defined by $$(x,y)\to (\varphi (x),r(x)y)$$ gives a bundle map
from $\xi|U$ to $\xi|V$ which is a lift of $\varphi$. This map is
independent of the trivializations of $\xi$ over $U$ and $V$ we
chose. Thus we have a bundle map from $N$ to $N$ which is a lift of
$\varphi$.

Thus by the naturality of Euler class, we have $\varphi^* e(\xi) =
e(\xi)$.  Note that by Theorem~\ref{main-exp}, $\varphi^*$ is an
expanding map on $H^q(X)$. So $1$ cannot be an eigenvalue of
$\varphi^*$, consequently $\varphi^* e(\xi) = e(\xi)$ implies
$e(\xi)=0$.

(b) In our current situation since the euler class  $e=0$, the Gysin
sequence splits into short exact sequence \begin{equation} 0 \to
H_{l-q+1} (X) \stackrel{\rho}\to H_l(\partial N) \to H_{l}(X) \to
0.\label{split}\end{equation} Then $H_l(\partial N) \cong H_l(X)
\oplus H_{l-q+1}(X)$ as all the homology groups here are over $\RR$,
and thus are vector spaces.

%Consider the commutative diagram
%$$\begin{CD} H_{l+1}(N, \partial N) @>\partial>> H_l(\partial N) @>>> H_l(N) @>>>H_l(N, \partial N) @ >\partial>> H_{l-1}(\partial N) \\
% @VV\Phi_*V   @| @VV\pi_*V @VV\Phi_*V @|\\
%H_{l-q+1}(X) @>>> H_l(\partial N) @>>> H_l(X) @>\cap e >> H_{l-q}(X)
%@>>> H_{l-1}(\partial N)\end{CD}.$$ The first line is the long exact
%sequence for the pair $(N, \partial N)$. The vertical map $\pi_*$ is
%an isomorphism as $N\simeq X$. The vertical map $\Phi_*$ is also an
%isomorphism, called \textit{Thom isomorphism}. Therefore the second
%line is also a long exact sequence, which is named as \textit{Gysin
%sequence}. Moreover the map from $H_l(X)$ to $H_{l-q}(X)$ is defined
%by $z\mapsto  e\cap z$, where $e$ is the Euler class of the disk
%bundle $\xi$, $\cap$ means cap product.

%In our current situation, $e=0$, so Gysin sequence splits into short
%exact sequence \begin{equation} 0 \to H_{l-q+1} (X) \to H_l(\partial
%N) \to H_{l}(X) \to 0.\label{split}\end{equation} Then $H_l(\partial
%N) \cong H_l(X) \oplus H_{l-q+1}(X)$ as all the homology groups here
%are over $\RR$, and thus are vector spaces.

(c) We need to understand more about the map $\rho: H_{l-q+1}(X) \to
H_l(\partial N)$ in (\ref{split}). According to the geometric
interpretation of cap product, for any $(l-q+1)$-cycle $c$ in $X$,
$\Phi^{-1}_{*}([c])$ is exactly represented by $\pi^{-1}(c)$, where
$\Phi_*: H_{l+1}(N, \partial N) \to H_{l-q+1}(X)$ is the Thom
isomorphism. Its boundary is $(\pi|\partial N)^{-1}(c)$, which
represents $\rho([c]) \in H_l(\partial N)$. Now the result follows
from the injectivity of $\rho$.
\end{proof}

\medskip
\begin{lemma} \label{shrink}
Let  $f: X=B\tilde \times S^q \to Y$ be a map, where $B$ is a finite
CW-complex and $Y$ is a $K(\pi, 1)$ space. If $q \geq 2$, then $f$
is homotopic to $\bar f \circ \pi$, where $\pi: B\tilde \times S^q
\to B$ is the projection, and $\bar f: B\to Y$.
\end{lemma}

\begin{proof} First we show by induction that $f$ can be
extended to a map  $\tilde f: B\tilde \times D^{q+1} \to Y$ where
$\partial (B\tilde \times D^{q+1})=B\tilde \times S^{q}$.

Let $B_i$ be the $i$-th dimensional skeleton of $B$, and $\tilde
X_i=X\cup B\tilde \times D^{q+1}|_{B_i}$, where $i= 0, 1,\cdots,
\dim B$. Since $q\ge 2$ and $Y$ is $K(\pi, 1)$, clearly we can
extend $f$ to $\tilde f|\tilde X_0$. Suppose we have extended $f$ to
$\tilde f|\tilde X_{i-1}$. Then for each $i$-cell $\Delta_i$ in $B$,
$\tilde f|$ has defined in $\partial (\Delta_i\times
D^{q+1})=(\partial \Delta_i\times D^{q+1}\cup \Delta_i\times S^{q})
\cong S^{i+q}$, and still by the $K(\pi,1)$ property of $Y$, we may
extend $\tilde f|\partial (\Delta_i\times D^{q+1})$ to $
\Delta_i\times D^{q+1}$. (Here we just write $\Delta_i\times
D^{q+1}$ because any disk bundle on $\Delta_i$ is trivial.)
Therefore $\tilde f$ can be extended to $\tilde X_i$ after finitely
many such steps.

Then the lemma follows easily. \end{proof}

\subsection{$\Omega(f)=\textit{DE
attractors}$ implies $M=\QQ \textit{-homology sphere}$}

Now we are ready to prove

\noindent \textbf{Theorem \ref{main1}.} \  \textit{If there exists a
diffeomorphism $f: M \to M$ on a closed, oriented $n$-manifold $M$
such that $\Omega(f)$ consists of finitely many oriented $(\pm)$
attractors derived from expanding maps, then $M$ is a rational
homology sphere. Moreover all those attractors are of type $(n-2,
2)$}.

\medskip
If $\Omega(f)$ consists of finitely many DE attractors,  we know
$\Omega(f)$ must be the union of two disjoint DE attractors $S_1$
and $S_2$, one is the attractor of $f$ and the other is the
attractor of $f^{-1}$ (see [JNW] Lemma 1). Suppose the two defining
disk bundles of attractors are $N_1\cong X_1^{p_1} \tilde{\times}
D^{q_1}$ and $N_2 \cong X_2^{p_2} \tilde{\times} D^{q_2}$
respectively. We have
$$S_1=\bigcap_{h=1}^{\infty} f^h(N_1), \  S_2=\bigcap_{h=1}^{\infty} f^{-h}(N_2),$$

\medskip

\begin{lemma}
\label{onebundle} Suppose an orientable manifold $N$ is a disk
bundle on a closed oriented manifold $X$. Let the embedding $e: N
\to {\rm Int} (N)$ be the lift of an expanding map on $X$,
$N'=e(N)$, $K=N \backslash {\rm Int}(N')$. For any integer $l$, the
two maps $H_l(\partial N)
 \to H_l(K)   \leftarrow H_l(\partial N'),$ induced by the inclusions $\partial N \subset K,
\partial N' \subset K $ respectively are isomorphisms. (cf. Figure
1)
\end{lemma}

\begin{figure}[htbp]
\begin{center}
\resizebox{!}{5cm}{
\includegraphics*[0mm, 0mm][70mm,80mm]{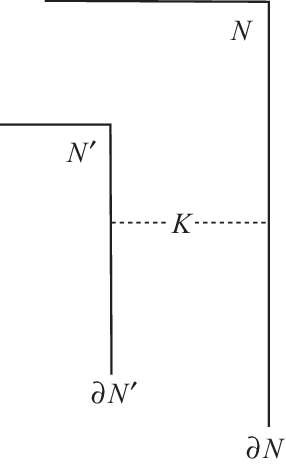}
}
\end{center}
\caption{\label{mvone} The sketch of self embedding of a disk
bundle}
\end{figure}

\begin{proof} First, consider the Mayer-Vietoris long exact sequence for the pair
$(N',K)$. We know $$H_{l}(N') \oplus H_{l}(K) \stackrel{\psi}\to
H_{l}(N)\stackrel{\partial}\to H_{l-1}(\partial N')$$ is exact. $N
\simeq X$ implies $H_*(N)\cong H_*(X)$ through the projection map.
Then it follows from Lemma \ref{expandhomo} that the map from
$H_{l}(N')$ to $H_{l}(N)$ is surjective and thus the first map
$\psi$ is surjective. By the exactness, the second map $\partial$
must be $0$ for any $l$ and we get the short exact sequence
$$0 \to H_l(\partial N') \xrightarrow{\varphi=(i_1, i_2)}  H_l(N') \oplus
H_l(K) \xrightarrow{\psi=j_1-j_2} H_l(N) \to 0.$$ By Lemma
\ref{expandhomo}, $j_1$ is an isomorphism, then according to Lemma
\ref{exact}, the map $i_2: H_l(\partial N') \to H_l(K)$ is also an
isomorphism. Half is done.

\bigskip
Next, consider the long exact sequence of the pair $(K, \partial
N')$:
$$\cdots \rightarrow  H_l(\partial N')
\stackrel{i_2}\rightarrow H_l(K) \rightarrow H_l(K, \partial
N')\rightarrow  H_{l-1}(\partial N') \stackrel{i_2}\rightarrow
H_{l-1}(K)\cdots$$ Note that $i_2$ involved are isomorphisms, so
$H_l(K,
\partial N') =0$. Then by Poicar\'{e} duality and algebraic duality, we have
$$0 = H_l(K,\partial N') =H^{n-l}(K, \partial N)= {\rm
Hom}(H_{n-l}(K,
\partial N), \RR).$$ Hence $H_{n-l}(K,
\partial N)=0$ for all $l$, or equivalently
$H_l(K, \partial N) =0$ for any integer $l$.

\bigskip
Finally, consider the long exact sequence for the pair $(K,
\partial N)$:
$$\cdots \rightarrow H_{l+1}(K, \partial N) \rightarrow  H_l(\partial N) \rightarrow H_l(K) \rightarrow H_l(K, \partial
N)\rightarrow \cdots.$$ It then follows from $H_{l+1}(K, \partial
N)=H_{l}(K, \partial N)=0$ that $H_l(\partial N) \to H_l(K)$ is an
isomorphism.
\end{proof}

\bigskip
The setting in the above lemma appears in a single attractor, while
the next lemma reveals the homological relations between the two DE
attractors in the non-wandering set. Recall that our assumption is
$S_1=\bigcap_{h=1}^{\infty} f^h(N_1), \  S_2=\bigcap_{h=1}^{\infty}
f^{-h}(N_2).$ Then $\bigcup_{h=1}^{\infty} f^{-h}({\rm Int} N_1) =
M- S_2$, it follows that $f^k(\partial N_2) \subset N_1$, $\ M
\backslash N_1 \subset f^k(N_2)$, for some large integer $k$.
Without loss of generality, we can substitute $N_2$ by $f^k(N_2)$.
Then $$M= N_1 \cup N_2 , \  P:= N_1 \cap N_2 \ {\rm with}\
\partial P= \partial N_1 \cup \partial
N_2.$$ Say $n=\dim M$, then the base manifold $X_1$ of $N_1$ is at
most $n-2$ dimensional (cf. Remark~\ref{pq3}). Hence
$H_{n-1}(N_1)=0$, which implies that $\partial N_2$ separates $N_1$
into two components. One of them is $P$ which contains $\partial
N_1$, while the other part $P'$ should contain $f^k(N_1)$ for some
large integer $k$. Set $N'=f^k(N_1)$, $K=N_1 \backslash {\rm
Int}(N')$. All the notations introduced above are illustrated in
Figure~\ref{mvtwo}.

\begin{figure}[htbp]

\begin{center}
\resizebox{!}{5cm}{
\includegraphics*[0mm, 0mm][60mm,80mm]{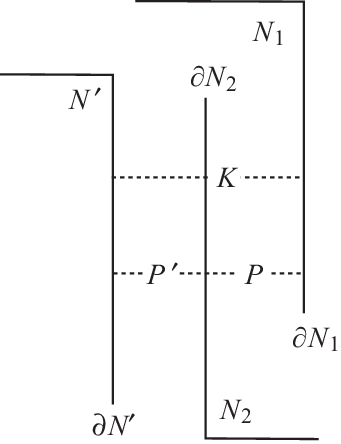}
}
\caption{\label{mvtwo} The sketch of two intersecting disk
bundles}
\end{center}
\end{figure}

\begin{lemma}
\label{twobundles} For any integer $l$, the two maps $H_l(\partial
N_1)
 \to H_l(P)   \leftarrow H_l(\partial N_2),$ induced by the inclusions $\partial N_1 \subset P,
\partial N_2 \subset P $ respectively are isomorphisms.

In particular, $\beta_l(\partial N_1)= \beta_l(\partial N_2)$ for
any integer $l$.
\end{lemma}

\begin{proof}

Applying the Mayer-Vietoris sequence for the pair $(P',P)$, we get
the exact sequence
$$
H_l(\partial N_2) \xrightarrow{(i_1, i_2)} H_l(P') \oplus H_l(P)
\xrightarrow{ j_1 -j_2} H_l(K).
$$

The inclusion $\partial N_1 \subset P \subset K$ induces
$H_l(\partial N_1) \to H_l(P) \stackrel{j_2}\to H_l(K)$. By lemma
\ref{onebundle}, this composition is an isomorphism. Hence $j_2$ is
surjective. Similarly, it follows from the inclusion $\partial N'
\subset P' \subset K$ and Lemma \ref{onebundle} that $j_1$ is
surjective. Moreover, by Lemma \ref{exact}(d),(e), $i_2$ is also
surjective. Then the surjectivity of $i_2$ and $j_2$ gives the
inequality
\begin{equation}\label{ineq1}
\dim H_l(\partial N_2) \geq \dim H_l (P) \geq \dim H_l(K) =\dim
H_l(\partial N_1).\end{equation} If we consider the diffeomorphism
$f^{-1}$ instead, then by symmetry, the above inequality turns
 into \begin{equation}\label{ineq2} \dim H_l(\partial
N_1) \geq \dim H_l(\partial N_2).\end{equation} Combining
(\ref{ineq1}) and (\ref{ineq2}), we get
\begin{equation}\label{alleq}
\dim H_l(\partial N_1)= \dim H_l (P)=\dim H_l(K)=\dim H_l(\partial
N_2).\end{equation}

Now all the homology groups in $H_l(\partial N_1) \to H_l(P)
\stackrel{j_2}\to H_l(K)$ have the same dimension and the
composition is an isomorphism by Lemma~\ref{onebundle}, thus the map
$H_l(\partial N_1) \to H_l(P)$ has to be an isomorphism. By symmetry
$H_l(\partial N_2) \to H_l(P)$ is also an isomorphism.
\end{proof}

\bigskip
Consider the Mayer-Vietoris sequence for the pair $(N_1, N_2):$
\begin{equation}\label{star}
H_l(P) \xrightarrow{\varphi_l=(r_1, r_2)} H_l(N_1) \oplus
H_l(N_2) \xrightarrow{\psi_l= s_1 -s_2} H_l(M).
\end{equation}

\begin{lemma}\label{non-surjective}
$\varphi_l$ is surjective for all $l>0$.
\end{lemma}

\begin{proof}
Otherwise $\psi_l$ is not zero for some $l>0$, and then at least one
$s_i$, say $s_1$, is not a zero map. As $f(N_1)\subset N_1$, we have
the following commutative diagram

$$\begin{CD}
 H_l(N_1; \ZZ) @>s_1>> H_l(M; \ZZ) \\
 @V(f|N_1)_*VV   @VVf_*V \\
H_l(N_1; \ZZ) @>>s_1> H_l(M; \ZZ) \end{CD}.$$

Note that $f|N_1$ is a lift of an expanding map on the base manifold
$X_1$ whose induced map on $H_l(X_1; \ZZ)$ is expanding by Theorem
\ref{main-exp} for $l>0$. Hence $(f|N_1)_*: H_l(N_1; \ZZ) \to
H_l(N_1; \ZZ)$ is expanding as well. On the other hand, it follows
from $f: M\to M$ is a diffeomorphism that $f_*: H_l(M; \ZZ) \to
H_l(M; \ZZ)$ is an isomorphism. But $s_1$ is non-zero,  which
contradicts to Lemma \ref{expend5}.
\end{proof}

\bigskip
\noindent {\bf Proof of Theorem \ref{main1}:} Recall that the two
defining disk bundles are $$N_1\cong X_1^{p_1} \tilde{\times}
D^{q_1}, N_2 \cong X_2^{p_2} \tilde{\times} D^{q_2},$$then
$n=p_1+q_1=p_2+q_2$. By symmetry, assume that $2\le q_1\le q_2$.
then $n-2\ge p_1\ge p_2 \ge 1$.

\medskip
\underline{Case 1}: $q_2\ge 3$.

Since $X_j$ is a closed orientable manifold of dimension $p_j$, we
have $\beta_{p_j}(X_j)=1$, ($j=1,2$). Set $l=p_1$, then
$1=\beta_l(X_1) \ge \beta_l(X_2)$.

Consider the diagram $$ \xymatrix{
  H_l(P) \ar[r]^<<<<<<<{\varphi_l}  & H_l(N_1)\oplus H_l(N_2)       \\
  U_2 \subset H_l(\partial N_2) \ar[u]_{i_2}^{\cong} \ar[ur]_{\ \ \phi_2}},
  \label{tridiag}$$
where $\varphi_l$ is defined in (\ref{star}),  $i_2$ is induced by
inclusion $\partial N_2 \subset P$, $U_2$ is given by
Lemma~\ref{euler0} (c), and $\phi_2= \varphi_l \circ i_2$.

Since $X_j$ admits an expanding map, the universal cover of $X_j$ is
contractible by Theorem~\ref{exp} (a), hence  $N_j \simeq X_j$ is a
$K(\pi, 1)$ space. Note that $\phi_2$ is induced by inclusion $g_j:
\partial N_2\to N_j$ . According to Lemma~\ref{euler0}
(c), $U_2$ has a basis $\{ [c_1 \tilde\times S^{q_2-1}], $ $[c_2
\tilde\times S^{q_2-1}], \cdots\}$, where $c_1, c_2, \cdots$ are
$l-q_2+1$ cycles in $X_2$. By Lemma~\ref{shrink}, as $q_2-1\geq 2$,
$g_j$ can be homotoped so that the image of $c_1 \tilde\times
S^{q_2-1}$, $c_2 \tilde\times S^{q_2-1}$, $\cdots$ under $g_j$ are
all of dimension $l-q_2+1$,($j=1,2$). Thus $\phi_2$ maps all these
basis elements to zero in both $H_l(N_1)$ and $H_l(N_2)$, in other
words, $U_2 \subset {\rm ker} \phi_2$.

It follows from Lemma~\ref{euler0} that $H_l(\partial N_2) \cong
H_l(X_2) \oplus H_{l-q_2+1}(X_2)$ while $\dim U_2= \dim
H_{l-q_2+1}(X_2)$. Thus $$\dim {\rm Im} \phi_2 \leq \beta_l(X_2) <
\beta_l(X_1) + \beta_l(X_2)= \beta_l(N_1) + \beta_l(N_2).$$
Consequently $\phi_2$ cannot be surjective.

On the other hand, $i_2$ is an isomorphism by
Lemma~\ref{twobundles}, $\varphi_l$ is surjective by
Lemma~\ref{non-surjective}, so $\phi_2$ must be surjective. We
derive a contradiction.

\underline{Case 2}: $q_1=q_2= 2$.

Consider
$$H_l(X_j)\oplus H_{l-1}(X_j) \xrightarrow{\cong}H_l(P)
\xrightarrow{\varphi_l} H_l(N_1) \oplus H_l(N_2) \xrightarrow{\cong}
H_l(X_1)\oplus H_l(X_2), j=1,2,$$ where the left isomorphism comes
from Lemma~\ref{euler0} (b) and Lemma~\ref{twobundles}, and the
right isomorphism comes from $N_j\simeq X_j$.

The surjectivity of $\varphi_l$ from Lemma~\ref{non-surjective}
implies that \begin{equation} \beta_l(X_j) + \beta_{l-1}(X_j) \geq
\beta_l(X_1) + \beta_l(X_2), j=1,2,\label{dimineq}\end{equation} or
equivalently $\beta_{l-1}(X_1)\ge \beta_{l}(X_2)$ and
$\beta_{l-1}(X_2)\ge \beta_{l}(X_1)$. Let $l$ goes from $1$ to
$n-2$, we get
$$1=\beta_0(X_j)\ge \cdots \ge \beta_{n-3}(X_2)\ge \beta_{n-2}(X_1)=1$$
and  $$1=\beta_0(X_k)\ge \cdots\ge \beta_{n-3}(X_1)\ge
\beta_{n-2}(X_2)=1$$ Therefore $\beta_i(X_1)=\beta_i(X_2)=1$ for
$i=0,...,n-2$.

Now the equality holds in $(\ref{dimineq})$ for $l=1, \cdots , n-2$,
hence in the whole Mayer-Vietoris sequence of $(N_1, N_2)$
$(\ref{star})$, $ \varphi_l$ is an isomorphism, $\psi_l = 0$ for
$l=1, \cdots , n-2$, therefore
$$H_{n-1}(M; \RR)=\cdots=H_1(M;\RR)=0.$$ It follows from the universal
theorem for homology that $$H_{n-1}(M; \QQ)=\cdots=H_1(M;\QQ)=0,$$ i.e. $M$ is a rational
homology sphere. \hfill $\Box$

\bigskip
\noindent {\bf Proof of Corollary \ref{main2}.} Suppose $f: M \to M$
is a diffeomorphism such that $\Omega(f)$ consists of finitely many
oriented DE attractors. Let $N_1$, $N_2$ be defined as in the
beginning of the proof of Theorem \ref{main1}.

(b) By Theorem \ref{main1}, $N_1\cong T^{n-2} \tilde{\times} D^{2}$
and $N_2 \cong T^{n-2} \tilde{\times} D^{2}$. Thus
$\beta_1(N_1)=\beta_1(N_2)=n-2$, and by Lemma~\ref{euler0} (b),
$\beta_1(\partial N_1)=\beta_1(\partial N_2)=n-1$. Now we have
$$\RR^{n-1}=H_1(P)
\xrightarrow{\varphi_1=(r_1, r_2)} H_1(N_1) \oplus
H_1(N_2)=\RR^{2n-4}.$$ Then $\varphi_1$ can not be surjective when
$n\ge 4$, which contradicts to Lemma \ref{non-surjective}. We proved
(b).

(a) When $n=4$, by Theorem \ref{main1}, the dimension of the base
manifolds $X_j$ must be 2. By Theorem \ref{class},
 $X_j$ is covered by Nil-manifold. Since the only orientable closed
2-manifold covered by Nil-manifold is torus, $j=1,2$. So (a) follows
from (b).

(c) is included in (a) and Theorem~\ref{main1}.
 \hfill $\square$

\textbf{Acknowledgement.}\ \ The first three authors are partially
supported by grant No.10631060 of the National Natural Science
Foundation of China, the second author is also partially supported
by NSFC project  60603004. The last author would like to thank
Professor R. Kirby for his continuous support and encouragement.

\bibliographystyle{amsalpha}

\begin{thebibliography}{FW}

\addcontentsline{toc}{section}{Reference}

\setlength{\itemsep}{0ex}

\bibitem[AA]{AA} A. L. Auslander, \textit{Bieberbach's theorems on space groups and discrete uniform subgroups
of Lie groups}, Annals of Math. \textbf{71} (1960), 579-590.


\bibitem[CE]{CE} C.Chevalley, S. Eilenberg, \textit{Cohomology theory
of Lie groups and Lie algebras}, Trans. Amer. Math. Soc. \textbf{63}
(1948), 85-124.

\bibitem[ES]{ES} D. Epstein, M. Shub, \emph{Expanding endomorphisms of flat
manifolds}, Topology \textbf{7}(1968), 139--141.

\bibitem[FW]{FW} J. Franks; B. Williams, \emph{Anomalous Anosov flows}. Global theory of
dynamical systems. 158--174, Lecture Notes in Math. \textbf{819},
Springer, Berlin, 1980.

\bibitem[Gr]{Gr} M. Gromov, \emph{Groups of polynomial growth and expanding
maps.} Inst. Hautes ¨¦tudes Sci. Publ. Math. No.
\textbf{53}(1981), 53-73.

\bibitem[HS]{HS} P.J. Hilton, U. Stammbach, \textit{A course in
homological algebra}, Grad. Texts in Math. Vol \textbf{4},
Springer-Verlag, 1971.

\bibitem[JNW]{JNW} B. J. Jiang, Y. Ni,
S. C. Wang, \emph{3-manifolds that admit knotted solenoids as
attractors}, Trans. Amer. Math. Soc. \textbf{356} (2004),
4371-4382.

\bibitem[JWZ]{JWZ} B. Jiang, S. Wang, H. Zheng, {\it No embeddings of solenoids
into surfaces.} Proc. Amer. Math. Soc. \textbf{136}(2008), no. 10,
3697--3700.


\bibitem[Ma]{Ma} J. Mather, \emph{Characterization of Anosov
diffeomorphisms}, Nederl. Akad. van Wetensch. Proc. Ser. A,
Armsterdam 71 = Indag. Math. \textbf{30}(1968), No.5 .


\bibitem[No]{No} K. Nomizu, \emph{On the cohomology of compact homogeneous spaces of
nilpotent Lie groups. Ann}. of Math. (2) 59, (1954). 531--538.


\bibitem[Sh1]{Sh1} M. Shub, \textit{Endomorphisms of compact
differentiable manifolds}, Amer. J. Math. \textbf{91}(1969),
175-199.

\bibitem[Sh2]{Sh2} M. Shub, \textit{Expanding maps} 1970 Global Analysis (Proc. Sympos.
Pure Math., Vol. XIV, Berkeley, Calif., 1968) 273--276

\bibitem [Sm]{Sm} S. Smale, {\it Differentiable dynamical systems},
    Bull. Amer. Math. Soc. \textbf{73} (1967), 747--817.

\bibitem [Sp]{Sp} E.H. Spanier, {\it Algebraic topology}, Springer-Verlag, 1966.

\bibitem [Vi]{Vi} L. Vietoris,
    {\it \"Uber den h\"oheren Zusammenhang kompakter R\"aume und eine Klasse von zusammenhangstreuen Abbildungen,}
    (German) Math. Ann. \textbf{97} (1927), no. 1, 454--472.

\bibitem[WZ]{WZ}  S.C. Wang, Q. Zhou,
{\it Any 3-manifold 1-dominates at most finitely many geometric
3-manifolds,} Math. Ann. \textbf{322}(2002), no. 3, 525-535.

\end{thebibliography}

\bigskip
\textsc{School of Mathematical Sciences, Peking University, Beijing
100871, P.R.China} \hskip 1true cm \textit{E-mail:}
\verb"dingfan@math.pku.edu.cn"

\bigskip
\textsc{Institute of Mathematics, Academy of Mathematics and Systems
science, Chinese Academy of Sciences,  Beijing 100080 P. R.  China}
\hskip 1true cm \textit{E-mail:} \verb"pjz@amss.ac.cn"

%\textsc{Institute of Mathematics Academia Sinica Beijing 100080 P.
%R.  China} \hskip 1true cm \textit{E-mail:} \verb"pjz@amss.ac.cn"

\bigskip
\textsc{School of Mathematical Sciences, Peking University, Beijing
100871, P.R.China} \hskip 1true cm\textit{E-mail:}
\verb"wangsc@math.pku.edu.cn"

\bigskip
\textsc{Department of Mathematics, University of California at
Berkeley, CA 94720, USA} \hskip 1true cm\textit{E-mail:}
\verb"jgyao@math.berkeley.edu"

\end{document}